\documentclass[11pt]{article}  

\usepackage{setspace}
\usepackage{float}
\usepackage{amsmath,amsfonts,amssymb}
\usepackage{amsthm}
\usepackage{graphicx}
\usepackage{epstopdf}
\usepackage{caption}
\usepackage{a4wide}
\usepackage{color}
\usepackage{url}

\theoremstyle{plain}
\newtheorem{theorem}{Theorem}
\newtheorem{lemma}{Lemma}

\theoremstyle{definition}
\newtheorem{definition}{Definition}
\newtheorem{example}{Example}
\newtheorem{remark}{Remark}

\usepackage{makeidx}
\makeindex

\newcommand\blfootnote[1]{%
	\begingroup
	\renewcommand\thefootnote{}\footnote{#1}%
	\addtocounter{footnote}{-1}%
	\endgroup
}

\begin{document}	
	
\title{Stability criteria of nonlinear generalized proportional\\ 
fractional delayed systems\blfootnote{This is a preprint version of a paper 
whose final form is published in 'Mathematical Analysis: Theory and Applications'.}}
	
\author{Hanaa Zitane$^{1,2}$\\
\url{https://orcid.org/0000-0002-7635-9963}\\	
\texttt{h.zitane@ua.pt,~h.zitane@uae.ac.ma}
\and Delfim F. M. Torres$^{1,}$\thanks{Corresponding author.}\\
\url{https://orcid.org/0000-0001-8641-2505}\\
\texttt{delfim@ua.pt}}

\date{$^{1}$\text{Center for Research and Development in Mathematics and Applications (CIDMA),}  
Department of Mathematics, University of Aveiro, 3810-193 Aveiro, Portugal\\[0.3cm] 
$^{2}$Department of Mathematics, Faculty of Sciences,\\
University of Abdelmalek Essaadi, B.P.~2121, Tetouan, Morocco}
	
\maketitle	
	
% -----------------------------------------------
	
\begin{abstract}
This work deals with the finite time stability of generalized 
proportional fractional systems with time delay. First, 
based on the generalized proportional Gr\"onwall inequality, 
we derive an explicit criterion that enables the system trajectories 
to stay within a priori given sets during a pre-specified time interval, 
in terms of the Mittag-Leffler function. Then, we investigate the finite 
time stability of nonlinear nonhomogeneous delayed systems by means 
of an approach based on H\"older's and Jensen's inequalities. 
Numerical applications are presented to illustrate the validity 
and feasibility of the developed results.

\medskip
		
		
\medskip
		
\end{abstract}
% -----------------------------------------------------------------------
	
\section{Introduction}

With the rapid development and progress of fractional calculus, 
there has been an increasing interest in the investigation 
of time delay fractional systems, since they allow 
to describe systems in which the rate of change depends 
not only on the present and delayed state but also on the 
whole past memory, as well as, they have diverse applications 
in science and engineering \cite{Gao,LiSun}. During the past decades, 
the finite time stability\index{finite time stability}  
of fractional delay systems has been widely 
studied \cite{Arthi,Du,YangWu,MyID:549}. The main approaches 
to analyze this kind of stability include 
the fractional Halanay inequality \cite{Nguyena}, 
H\"{o}lder's inequality \cite{DuJia}, 
Gr\"onwall's inequality \cite{Lazarevi,YangZhang}, 
delayed Mittag-Leffler type matrix functions \cite{Li}, 
and the weighted integral inequality \cite{Jia}.

Here, we consider the finite time stability\index{finite time stability} problem of a class 
of nonlinear generalized proportional fractional systems (GPFSs) 
with time delay:\index{time delay}
\begin{equation}
\label{system2}
\left\{
\begin{array}{ll}
{}^C \!D_{0}^{\alpha,\mu}y(t)
=\exp\left(\frac{\mu-1}{\mu}t\right) \left[Ay(t)+By(t-\tau)+f(t,y(t),y(t-\tau))\right], 
& t\in[0,T],\\
y(t)=\phi(t),  & t\in[-\tau,0],
\end{array}
\right.
\end{equation}
where $A$ and $B$ are constant $n\times n$ matrices, 
$T>0$ is a real number, $\tau>0$ is a time delay,\index{time delay} 
$\phi(\cdot)$ is a continuous function on $[-\tau,0]$,
$f: [0,T]\times\mathbb{R}^{n}\times\mathbb{R}^{n}\longrightarrow \mathbb{R}^{n}$  
is a given nonlinear continuous function with $f(t,0,0)=0$, 
and ${}^C \!D_{0}^{\alpha,\mu}$ is the Caputo generalized proportional 
fractional derivative (GPFD) as suggested by Jarad et al.~\cite{GPF}. 
This type of fractional operators has different interesting advantages, 
such as: it preserves the semigroup property, 
which is essential for solving certain complicated fractional systems; 
it possesses a nonlocal character; and tends to the original function 
and its derivative upon a limiting process \cite{Gronwall}. Moreover, 
it provides an undeviating generalization to the existing Caputo fractional derivative.
For further information about the GPFD, differential systems evolving under GPFDs 
and their applications, we refer the reader to \cite{Gronwall,Boucenna,Farman,Laadjal} 
and references therein.

The stability of GPFSs, with and without delay, has been seldom investigated 
\cite{Almeida,Bohner,LaadjalJ}. In \cite{Almeida}, the second method 
of Lyapunov is used to analyze the exponential and the Mittag-Leffler 
stability of fractional order systems evolving GPFDs. Also, in \cite{LaadjalJ}, 
sufficient conditions that ensure the stability of nonlinear hybrid fractional 
integro-differential equations with Dirichlet boundary conditions are provided
in the sense of Ulam--Hyers--Rassias. For the stability of time delay systems, 
it has been tackled only by means of the Lyapunov approach, where quadratic Lyapunov 
functions and their GPFD are applied to investigate the exponential and asymptotic 
stability of a scalar nonlinear integro-differential GPFS with bounded delays \cite{Bohner}.
This serves as inspiration for our current work.

We study the finite time stability problem for a class of nonlinear GPFSs 
with time delay.\index{time delay} The main contributions of this work can be summarized
as follows:
\begin{description}
\item[(i)] A simple sufficient condition is established for the stability of 
system~\eqref{system2} over a finite time interval in the homogeneous case 
by means of the Bellman--Gr\"onwall approach.

\item[(ii)] An explicit delay-dependent criterion is obtained to ensure 
the finite time stability of nonlinear time delay GPFSs by using 
H\"older's and Jensen's inequalities.

\item[(ii)] Compared to the existing work \cite{Bohner}, here we consider 
a more general and multi dimension class of GPFSs with delays. 
Furthermore, the stability is investigated based on non-Lyapunov approaches.		
\end{description}

% -----------------------------------------------------------------------

\section{Preliminaries}
\label{sec:2}

Throughout the text, $\Arrowvert \cdot \Arrowvert$ stands for 
the Euclidean norm and the spectral norm for a vector and a matrix, 
respectively. Also, the norm function with the initial condition is given by 
$$
\Arrowvert \phi \Arrowvert_{C}
= \underset{t \in [-\tau,0]}{\sup} \Arrowvert \phi (t)\Arrowvert.
$$ 

We begin by recalling the definitions of generalized proportional fractional operators 
that are employed throughout the manuscript. Further information about the generalized 
proportional calculus can be found in \cite{Gronwall,GPF,Laadjal}. 

\begin{definition}[See \cite{GPF}]
Let $\alpha>0$, $\mu \in(0,1]$, and $h \in C^{1}([0,T], \mathbb{R})$. 
The left generalized proportional fractional integral\index{generalized proportional fractional integral} 
\cite{Gronwall} of function $h$ is defined by
$$
{}^C\!I_{0}^{\alpha,\mu}h(t)=\dfrac{1}{\mu^{\alpha}\Gamma(\alpha)}
\int_{0}^{t}\exp\left(\frac{\mu-1}{\mu}(t-s)\right) (t-s)^{\alpha-1}h(s)\, \mathrm{d}s,
$$
where $\Gamma(\cdot)$ is Euler's Gamma function \cite{Sabatier}. 
\end{definition}

\begin{definition}[See \cite{GPF}] 
\label{GPF}
Let $\alpha\in(0,1)$ and $\mu \in(0,1]$. The left generalized
proportional fractional order derivative\index{generalized proportional fractional derivative}  
of a function 
$h \in C^{1}([0,T], \mathbb{R})$ is given by
$$
{}^C\!D_{0}^{\alpha,\mu}h(t)=\dfrac{1}{\mu^{1-\alpha}\Gamma(1-\alpha)}
\int_{0}^{t}\exp\left(\frac{\mu-1}{\mu}(t-s)\right) 
(t-s)^{-\alpha}D^{1,\mu}h(s)\, \mathrm{d}s,
$$
where $D^{1,\mu}h(t)=(1-\mu)h(t)+\mu h'(t)$.
\end{definition}

\begin{remark}
When $\mu=1$, the GPFD reduces to the Caputo 
fractional derivative~\cite{Sabatier}.	
\end{remark}

The existence result of a solution to system~\eqref{system2} 
is given by the following lemma.

\begin{lemma}[See \cite{Gronwall}] 
\label{Lemma:1}
A function $y : [-\tau, T]\longrightarrow\mathbb{R}^{n}$ 
is a mild solution\index{solution} of system \eqref{system2} 
if, and only if, it satisfies
\begin{equation}
\label{system-2}
\left\{
\begin{array}{ll}
y(t)=\phi(0)\exp\left(\frac{\mu-1}{\mu}t\right) 
+{}^C \!I_{0}^{\alpha,\mu}
\exp\left(\frac{\mu-1}{\mu}s\right) \left[  
Ay(t)+By(t-\tau)+f(t,y(t),y(t-\tau))\right], 
& t\in[0,T],\\
y(t)=\phi(t),  & t\in[-\tau,0].
\end{array}
\right.
\end{equation}
Moreover, if for any functions 
$y,z : [-\tau, T]\longrightarrow\mathbb{R}^{n}$
there exists a constant $L_{f}>0$ such that
\begin{equation}
\label{condLipch}
\Arrowvert f(t,y(t),y(t-\tau))- f(t,z(t),z(t-\tau))  \Arrowvert 
\leq L_{f}  \left(\Arrowvert  y(t)-z(t) \Arrowvert 
+ \Arrowvert  y(t-\tau)-z(t-\tau) \Arrowvert
\right),\quad t\in[0,T],
\end{equation}
then system~\eqref{system2} has a unique mild solution.
\end{lemma}

\begin{lemma}[Generalized proportional fractional Gr\"onwall 
inequality\index{generalized proportional fractional Gr\"onwall inequality} \cite{Gronwall}]
\label{Gronwall}
Suppose $\beta>0$, $\mu>0$, $g$ and $J$ are nonnegative 
and locally integrable functions on $[0,t_{f})$ $(t_{f}\leq \infty)$ 
and $h$ is a nonnegative, nondecreasing,
and continuous function on $[0,t_{f})$ satisfying 
$h(t) \leq L$, where $L$ is a constant. In addition, if
$$
g(t)\leq J(t)+ h(t)\displaystyle 
\int_{0}^{t}\exp\left(\frac{\mu-1}{\mu}(t-s)\right) (t-s)^{\alpha-1}g(s) \, \mathrm{d}s,
$$
then
$$
g(t)\leq J(t)+\displaystyle \int_{0}^{t} \left[
\sum_{n=1}^{+\infty}\dfrac{\left(h(t)\Gamma(\alpha)\right)^{n}}{\Gamma(n\alpha)}
\exp\left(\frac{\mu-1}{\mu}(t-s)\right) (t-s)^{n\alpha-1}J(s)\right]\, \mathrm{d}s,
\quad t\in[0,t_{f}].
$$
Moreover, if function $J$ is nondecreasing on $[0, t_{f})$, then
\begin{equation}
\label{GrenwalMittga}
g(t)\leq f(t)E_{\alpha}(h(t)\Gamma(\alpha)t^{\alpha}),
\quad t\in[0,t_{f}],
\end{equation}
where $E_{\alpha}(\cdot)$ is the Mittag-Leffler 
function\index{Mittag-Leffler 
	function} of one parameter \cite{Sabatier} given by
$$
E_{\alpha}(x)=\displaystyle
\sum_{k=0}^{+\infty}\dfrac{x^{k}}{\Gamma(\alpha k+1)},
\quad x\in\mathbb{C}.
$$
\end{lemma}

We conclude this section by the definition 
of finite time stability\index{finite time stability} of system \eqref{system2}.

\begin{definition}
\label{def1:stab}
For given positive numbers $c_{1}$ and $c_{2}$, $c_{1}\leq c_{2}$,	
system \eqref{system2} is finite time stable\index{finite time stable} with respect to 
$\{c_{1}, c_{2}, T \}$ if 
\begin{equation}
\Arrowvert \phi \Arrowvert_{C} \leq c_{1} \Rightarrow
\Arrowvert y(t) \Arrowvert < c_{2}, \quad \forall t \in [0,T].
\end{equation}
\end{definition}

\begin{remark} 
If we let $\mu=1$ in system \eqref{system2}, then one obtains 
the finite time stability definition of Caputo fractional 
systems with delays \cite{Lazarevi}.
\end{remark}	
	
% -----------------------------------------------------------------------

\section{Finite time stability of delayed GPFSs}
\label{sec:3}

In this section, we begin by studying the stability of the GPFS \eqref{system2} 
over a finite time interval in the homogeneous case $f=0$. 

\begin{theorem}
\label{theob}
System \eqref{system2} is finite time stable\index{finite time stable}
with respect to $\{c_{1}, c_{2}, T\}$, $c_{1}\leq c_{2}$, 
if the function~$f$ satisfies condition~\eqref{condLipch} and
\begin{equation}
\label{cnd1} 
\left(1
+\dfrac{(\Arrowvert A \Arrowvert + \Arrowvert B
\Arrowvert )t^{\alpha}}{\mu^{\alpha}\Gamma(\alpha+1)}\right)
E_{\alpha}\left(\dfrac{(\Arrowvert A \Arrowvert + \Arrowvert B
\Arrowvert)}{\mu^{\alpha}}t^{\alpha}\right)
\leq \dfrac{c_{2}}{c_{1}},
\end{equation}
for all $t \in [0,T]$.
\end{theorem}

\begin{proof}
From Lemma~\ref{Lemma:1}, the solution\index{solution} 
of system~\eqref{system2} can be written as  
\begin{multline*}	
y(t)=\phi(0)\exp\left(\frac{\mu-1}{\mu}t\right) 
+\dfrac{1}{\mu^{\alpha}\Gamma(\alpha)}\displaystyle \int_{0}^{t}
\exp\left(\frac{\mu-1}{\mu}(t-s)\right)\\
\times\exp\left(\frac{\mu-1}{\mu}s\right) (t-s)^{\alpha-1}
\left[ A(s)y(s)+B(s)y(s-\tau)\right] \, \mathrm{d}s
\end{multline*}
for all $t\in [0,T]$. Taking into account condition \eqref{condLipch} 
and using $\exp\left(\frac{\mu-1}{\mu}s\right) \leq 1$ 
for all $s\in [0,t]$, we obtain that
\begin{equation}
\label{impa}
\Arrowvert y(t)\Arrowvert
\leq\Arrowvert\phi(0)\Arrowvert+\dfrac{1}{\mu^{\alpha}
\Gamma(\alpha)}\displaystyle \int_{0}^{t}\exp\left(\frac{\mu-1}{\mu}(t-s)\right) (t-s)^{\alpha-1}
\left[\Arrowvert A \Arrowvert\Arrowvert y(s)\Arrowvert
+\Arrowvert B \Arrowvert \Arrowvert y(s-\tau)\Arrowvert\right] \, \mathrm{d}s.
\end{equation}
For all $t\in [0,T]$, let us consider $ z(t)
=\underset{0\leq \theta\leq t}{\sup}\Arrowvert y(\theta)\Arrowvert$. Then, one has 
\begin{equation}
\label{imp1}
\Arrowvert y(s-\tau)\Arrowvert \leq  z(s)+ \Arrowvert\phi\Arrowvert_{C}, 
\quad \forall s\in [0,t].
\end{equation}
Combining relation \eqref{imp1} and inequality \eqref{impa}, it implies that
\begin{equation}
\label{impb}
\Arrowvert y(t)\Arrowvert
\leq\Arrowvert\phi(0)\Arrowvert+\dfrac{1}{\mu^{\alpha}
\Gamma(\alpha)}\displaystyle\int_{0}^{t}\exp\left(\frac{\mu-1}{\mu}(t-s)\right) 
(t-s)^{\alpha-1}(\Arrowvert A \Arrowvert + \Arrowvert B
\Arrowvert )\left( z(s)
+ \Arrowvert\phi\Arrowvert_{C}\right) \, \mathrm{d}s.
\end{equation}
It follows that
\begin{equation}
\label{eq1}
\Arrowvert y(t)\Arrowvert\leq\Arrowvert\phi\Arrowvert_{C}
+\dfrac{(\Arrowvert A \Arrowvert + \Arrowvert B
\Arrowvert )t^{\alpha}}{\mu^{\alpha}
\Gamma(\alpha+1)}\Arrowvert\phi\Arrowvert_{C}
+\dfrac{(\Arrowvert A \Arrowvert + \Arrowvert B
\Arrowvert )}{\mu^{\alpha}\Gamma(\alpha)}
\displaystyle \int_{0}^{t}\exp\left(\frac{\mu-1}{\mu}(t-s)\right) 
(t-s)^{\alpha-1} z(s) \, \mathrm{d}s
\end{equation}
and, using the change of variable $x=t-s$, we get
\begin{equation}
\label{det1}
\Arrowvert y(t)\Arrowvert\leq\Arrowvert\phi\Arrowvert_{C}
+\dfrac{(\Arrowvert A \Arrowvert + \Arrowvert B
\Arrowvert )t^{\alpha}}{\mu^{\alpha}
\Gamma(\alpha+1)}\Arrowvert\phi\Arrowvert_{C}
+\dfrac{(\Arrowvert A \Arrowvert + \Arrowvert B
\Arrowvert )}{\mu^{\alpha}\Gamma(\alpha)}
\displaystyle \int_{0}^{t}\exp\left(\frac{\mu-1}{\mu}x\right) x^{\alpha-1}
 z(t-x) \, \mathrm{d}x.
\end{equation}
Moreover, by taking $t=\theta$ in \eqref{det1} with $\theta \in [0,t]$ 
and using $\theta^{\alpha}\leq t^{\alpha}$, we obtain
\begin{equation}
\label{eq2}
\Arrowvert y(\theta)\Arrowvert
\leq\left[1+\dfrac{(\Arrowvert A \Arrowvert + \Arrowvert B
\Arrowvert )t^{\alpha}}{\mu^{\alpha}
\Gamma(\alpha+1)}\right]\Arrowvert\phi\Arrowvert_{C}
+\dfrac{(\Arrowvert A \Arrowvert + \Arrowvert B
\Arrowvert )}{\mu^{\alpha}\Gamma(\alpha)}
\displaystyle \int_{0}^{\theta}\exp\left(\frac{\mu-1}{\mu}x\right) x^{\alpha-1}
 z(\theta-x) \, \mathrm{d}x.
\end{equation}
Since function $ z$ is nonnegative, it implies that 
$\displaystyle \int_{0}^{t}\exp\left(\frac{\mu-1}{\mu}x\right) x^{\alpha-1}
 z(t-x) \, \mathrm{d}x$ is an increasing function 
with respect to $t\geq0$, which yields
\begin{equation*}
\Arrowvert y(\theta)\Arrowvert\leq\left[1
+\dfrac{(\Arrowvert A \Arrowvert + \Arrowvert B\Arrowvert )
t^{\alpha}}{\mu^{\alpha}\Gamma(\alpha+1)}\right]
\Arrowvert\phi\Arrowvert_{C}+\dfrac{(\Arrowvert A \Arrowvert + \Arrowvert B
\Arrowvert )}{\mu^{\alpha}\Gamma(\alpha)}
\displaystyle \int_{0}^{t}\exp\left(\frac{\mu-1}{\mu}x\right) x^{\alpha-1}
 z(t-x)\, \mathrm{d}x.
\end{equation*}
Hence,
\begin{equation*}
 z(t)\leq\left[1+\dfrac{(\Arrowvert A \Arrowvert 
+ \Arrowvert B \Arrowvert) t^{\alpha}}{\mu^{\alpha}
\Gamma(\alpha+1)}\right]\Arrowvert\phi\Arrowvert_{C}
+\dfrac{(\Arrowvert A \Arrowvert + \Arrowvert B
\Arrowvert )}{\mu^{\alpha}\Gamma(\alpha)}
\displaystyle \int_{0}^{t}\exp\left(\frac{\mu-1}{\mu}(t-s)\right) (t-s)^{\alpha-1}
 z(s)\, \mathrm{d}s.
\end{equation*}
Denote $J(t)=\left[1+\dfrac{(\Arrowvert A \Arrowvert + \Arrowvert B
\Arrowvert )t^{\alpha}}{\mu^{\alpha}
\Gamma(\alpha+1)}\right]\Arrowvert\phi\Arrowvert_{C}$,
which is a nondecreasing function. By using Lemma~\ref{Gronwall} with 
$$
h(t)=\dfrac{\Arrowvert A \Arrowvert + \Arrowvert B
\Arrowvert}{\mu^{\alpha}\Gamma(\alpha)}, 
$$
one obtains
$$
\Arrowvert y(t)\Arrowvert
\leq \Arrowvert\phi\Arrowvert_{C} \left[1
+\dfrac{(\Arrowvert A \Arrowvert + \Arrowvert B
\Arrowvert )t^{\alpha}}{\mu^{\alpha}
\Gamma(\alpha+1)}\right]E_{\alpha}\left(
\dfrac{(\Arrowvert A \Arrowvert + \Arrowvert B
\Arrowvert )}{\mu^{\alpha}}t^{\alpha}\right).
$$
Therefore, by virtue of $\Arrowvert \phi \Arrowvert_{C} \leq c_{1}$ 
and~\eqref{cnd1}, it follows that
$\Arrowvert y(t) \Arrowvert < c_{2}$ 
for all $t \in [0,T]$.
\end{proof}

\begin{remark} 
If we let $\mu=1$ in system~\eqref{system2} and Theorem~\ref{theob}, 
then one retrieves the condition 
\begin{equation*}
\left(1+\dfrac{(\Arrowvert A \Arrowvert + \Arrowvert B
\Arrowvert) t^{\alpha}}{\Gamma(\alpha+1)}\right)
E_{\alpha}\left((\Arrowvert A \Arrowvert + \Arrowvert B
\Arrowvert) t^{\alpha}\right)
\leq \dfrac{c_{2}}{c_{1}}, 
\quad \forall t \in [0,T],
\end{equation*}
for the finite time stability of the Caputo 
fractional order time-delay system, 
established in~\cite{Lazarevi}.
\end{remark}

Now, we shall characterize a sufficient condition 
to ensure the finite time stability\index{finite time stability} 
of the GPFS~\eqref{system2}.

\begin{theorem}
\label{theoa}
System~\eqref{system2} is finite time stable\index{finite time stable} with 
respect to $\{c_{1}, c_{2},T\}$, $c_{1}\leq c_{2}$, 
if the function~$f$ satisfies condition~\eqref{condLipch} and
\begin{equation}
\label{cnd4}
^{r}\sqrt{\dfrac{3^{\frac{1}{\alpha}}r+(3^{\frac{1}{\alpha}}\psi+r\varphi
+\psi\varphi)\exp\left((\psi+r) t\right) }{r+\psi}}
\leq \dfrac{c_{2}}{c_{1}}, \quad \forall t \in [0,T],
\end{equation}
where
\begin{equation}
\label{Psi}
\psi=\dfrac{3^{\frac{1}{\alpha}}((\Arrowvert A \Arrowvert+L_{f})^{r}
+(\Arrowvert B \Arrowvert +L_{f})^{r}
\exp\left(-r\tau\right) )\omega^{r}}{\mu^{k}\Gamma^{r}(\alpha)},
\end{equation}
\begin{equation}
\label{phi}
\varphi=\dfrac{3^{\frac{1}{\alpha}}(\Arrowvert B \Arrowvert 
+L_{f})^{r}(1-\exp\left(-\tau q\right) )}{r\mu^{k}\Gamma^{r}(\alpha)}\omega^{r},
\end{equation}
$k=1+\alpha$, $r=1+\dfrac{1}{\alpha}$, and 
$\omega=\left(\dfrac{\Gamma(\alpha^{2})}{k^{\alpha^{2}}}\right)^{\dfrac{1}{k}}$.
\end{theorem}

\begin{proof} 
According to Lemma~\ref{Lemma:1}, the mild solution\index{solution}  
of system~\eqref{system2} is given by  
\begin{multline*}
y(t)=\phi(0)\exp\left(\frac{\mu-1}{\mu}t\right) 
+\dfrac{1}{\mu^{\alpha}\Gamma(\alpha)}
\displaystyle \int_{0}^{t}\exp\left(\frac{\mu-1}{\mu}(t-s)\right) 
\exp\left(\frac{\mu-1}{\mu}s\right) (t-s)^{\alpha-1}\\
\times \left[  
A(s)y(s)+B(s)y(s-\tau)+f(s,y(s),y(s-\tau))\right] \, \mathrm{d}s.
\end{multline*}
By virtue of condition~\eqref{condLipch} with $f(s,0,0)=0$, it follows that
\begin{equation*}
\Arrowvert y(t)\Arrowvert\leq\Arrowvert\phi(0)\Arrowvert
+\dfrac{1}{\mu^{\alpha}\Gamma(\alpha)}\displaystyle\int_{0}^{t}(t-s)^{\alpha-1}
\left[(\Arrowvert A \Arrowvert+L_{f})\Arrowvert y(s)\Arrowvert
+(\Arrowvert B \Arrowvert +L_{f})\Arrowvert y(s-\tau)\Arrowvert\right] \, \mathrm{d}s,
\end{equation*}
which implies that
\begin{multline*}
\Arrowvert y(t)\Arrowvert
\leq\Arrowvert\phi(0)\Arrowvert+\dfrac{\Arrowvert A \Arrowvert
+L_{f}}{\mu^{\alpha}\Gamma(\alpha)}\displaystyle\int_{0}^{t}(t-s)^{\alpha-1}
\exp\left(s\right) \exp\left(-s\right) 
\Arrowvert y(s)\Arrowvert\, \mathrm{d}s\\
+\dfrac{\Arrowvert B \Arrowvert +L_{f}}{\mu^{\alpha}\Gamma(\alpha)}
\displaystyle\int_{0}^{t}(t-s)^{\alpha-1}\exp\left(s\right) \exp\left(-s\right) 
\Arrowvert y(s-\tau)\Arrowvert \, \mathrm{d}s.
\end{multline*}
By H\"older's inequality \cite{Holder}, we get
\begin{multline}
\label{imp2}	
\Arrowvert y(t)\Arrowvert
\leq\Arrowvert\phi(0)\Arrowvert+\dfrac{\Arrowvert A \Arrowvert
+L_{f}}{\mu^{\alpha}\Gamma(\alpha)}\left(\displaystyle
\int_{0}^{t}(t-s)^{k(\alpha-1)}\exp\left(ks\right) \, 
\mathrm{d}s\right)^{\dfrac{1}{k}} \times \left( 
\displaystyle\int_{0}^{t} \exp\left(-rs\right) \Arrowvert y(s)\Arrowvert^{r}\, 
\mathrm{d}s\right)^{\dfrac{1}{r}}\\
+\dfrac{\Arrowvert B \Arrowvert +L_{f}}{\mu^{\alpha}\Gamma(\alpha)}\left(
\displaystyle\int_{0}^{t}(t-s)^{k(\alpha-1)}\exp\left(ks\right) \, 
\mathrm{d}s\right)^{\dfrac{1}{k}} \times \left( 
\displaystyle\int_{0}^{t} \exp\left(-rs\right) \Arrowvert y(s-\tau)\Arrowvert^{r}\,
\mathrm{d}s\right)^{\dfrac{1}{r}}
\end{multline}
with $k=1+\alpha$ and $r=1+\dfrac{1}{\alpha}$. Also, one has 
\begin{equation}
\label{imp3}
\displaystyle\int_{0}^{t}(t-s)^{k(\alpha-1)}\exp\left(ks\right) \, \mathrm{d}s 
\leq \dfrac{\exp\left(kt\right)\Gamma(k(\alpha-1)+1)}{k^{k(\alpha-1)+1}}
=\dfrac{\Gamma(\alpha^{2})\exp\left(kt\right)}{k^{\alpha^{2}}}.
\end{equation}
Combining inequalities~\eqref{imp2} and~\eqref{imp3} yields
\begin{multline*}
\Arrowvert y(t)\Arrowvert
\leq \Arrowvert\phi(0)\Arrowvert+\dfrac{(\Arrowvert A \Arrowvert
+L_{f})\omega \exp\left(t\right) }{\mu^{\alpha}\Gamma(\alpha)} \left( 
\displaystyle\int_{0}^{t} \exp\left(-rs\right) \Arrowvert y(s)\Arrowvert^{r}\, 
\mathrm{d}s\right)^{\dfrac{1}{r}}\\
+\dfrac{(\Arrowvert B \Arrowvert +L_{f})\omega \exp\left(t\right) }{\mu^{\alpha}
\Gamma(\alpha)} \left( \displaystyle\int_{0}^{t}  
\exp\left(-rs\right) \Arrowvert y(s-\tau)\Arrowvert^{r}\, 
\mathrm{d}s\right)^{\dfrac{1}{r}}
\end{multline*}
with $\omega=\left(\dfrac{\Gamma(\alpha^{2})}{k^{\alpha^{2}}}\right)^{\dfrac{1}{k}}$. 
This implies that
\begin{multline}
\label{eq:pr:th1}
\Arrowvert y(t)\Arrowvert
\leq\Arrowvert\phi(0)\Arrowvert
+\dfrac{(\Arrowvert A \Arrowvert+L_{f})\omega \exp\left(t\right) }{\mu^{\alpha}
\Gamma(\alpha)} \left( \displaystyle\int_{0}^{t} \exp\left(-rs\right) 
\Arrowvert y(s)\Arrowvert^{r}\, \mathrm{d}s\right)^{\dfrac{1}{r}}\\
+\dfrac{(\Arrowvert B \Arrowvert +L_{f})\omega \exp\left(t\right) }{\mu^{\alpha}\Gamma(\alpha)} 
\left( \displaystyle\int_{-\tau}^{t}   
\exp\left(-r(s+\tau)\right) \Arrowvert y(s)\Arrowvert^{r}\, \mathrm{d}s\right)^{\dfrac{1}{r}}.	
\end{multline}
Now, by applying Jensen's inequality \cite{Kuczma} 
to inequality~\eqref{eq:pr:th1}, we obtain
\begin{multline*}
\Arrowvert y(t)\Arrowvert^{r}
\leq 3^{\frac{1}{\alpha}}\left[\Arrowvert\phi(0)\Arrowvert^{r}
+\dfrac{(\Arrowvert A \Arrowvert+L_{f})^{r}\omega^{r}\exp\left(rt\right) }{\mu^{k}
\Gamma^{r}(\alpha)} \left( \displaystyle\int_{0}^{t} 
\exp\left(-rs\right) \Arrowvert y(s)\Arrowvert^{r}\, \mathrm{d}s\right)\right.\\
+\left.\dfrac{(\Arrowvert B \Arrowvert +L_{f})^{r}\omega^{r}\exp\left(rt\right) }{\mu^{k}
\Gamma^{r}(\alpha)} \left( \displaystyle\int_{-\tau}^{t}       
\exp\left(-r(s+\tau)\right) \Arrowvert y(s)\Arrowvert^{r}\, \mathrm{d}s\right)\right].	
\end{multline*}
Then,
\begin{multline*}	
\Arrowvert y(t)\Arrowvert^{r}
\leq 3^{\frac{1}{\alpha}}\Arrowvert\phi(0)\Arrowvert^{r}
+\psi \exp\left(rt\right)  \displaystyle\int_{0}^{t} 
\exp\left(-rs\right) \Arrowvert y(s)\Arrowvert^{r}\, \mathrm{d}s\\
+\dfrac{3^{\frac{1}{\alpha}}(\Arrowvert B \Arrowvert 
+L_{f})^{r}\omega^{r} e^{q(t-\tau)}}{\mu^{k}\Gamma^{r}(\alpha)} 
\displaystyle\int_{-\tau}^{0}       
\exp\left(-rs\right) \Arrowvert y(s)\Arrowvert^{r}\, \mathrm{d}s,
\end{multline*}
where $\psi$ is defined in~\eqref{Psi}, which implies that
$$
\Arrowvert y(t)\Arrowvert^{r}\leq3^{\frac{1}{\alpha}}\Arrowvert\phi\Arrowvert_{C}^{r}
+\exp\left(rt\right) \Arrowvert\phi\Arrowvert_{C}^{r}\varphi 
+\psi \exp\left(rt\right)  \displaystyle\int_{0}^{t} 
\exp\left(-rs\right) \Arrowvert y(s)\Arrowvert^{r}\, \mathrm{d}s
$$
with $\varphi$ given by~\eqref{phi}. It follows that
$$
\exp\left(-rt\right) \Arrowvert y(t)\Arrowvert^{r}\leq\left(
3^{\frac{1}{\alpha}}\exp\left(-rs\right) +\varphi\right)\Arrowvert\phi\Arrowvert_{C}^{r}
+\psi \displaystyle\int_{0}^{t} \exp\left(-rs\right) \Arrowvert y(s)\Arrowvert^{r}\, \mathrm{d}s.
$$
Therefore, by applying Gr\"onwall's inequality, one gets
$$
\exp\left(-rt\right) \Arrowvert y(t)\Arrowvert^{r}\leq\left(
3^{\frac{1}{\alpha}}\exp\left(-rs\right) +\varphi\right)\Arrowvert\phi\Arrowvert_{C}^{r}
+\displaystyle\int_{0}^{t} \psi\left(3^{\frac{1}{\alpha}}\exp\left(-rs\right) 
+\varphi\right)\Arrowvert\phi\Arrowvert_{C}^{r}\exp\left(\psi(t-s)\right) \, \mathrm{d}s,
$$
and
\begin{equation}
\label{concl}
\Arrowvert y(t)\Arrowvert^{r}\leq
\dfrac{3^{\frac{1}{\alpha}}r+(3^{\frac{1}{\alpha}}\psi+r\varphi+\psi\varphi)
\exp\left((\psi+r) t\right) }{r+\psi}\Arrowvert\phi\Arrowvert_{C}^{r}.
\end{equation}
Hence, by combining  condition~\eqref{cnd4} and inequality~\eqref{concl}, 
we conclude with the finite time 
stability\index{finite time stability} of the GPFS~\eqref{system2}.
\end{proof}

% ---------------------------------------------------

\section{Applications}
\label{sec:4}

In this section, two numerical examples are provided 
to illustrate the effectiveness of the proposed results.

\begin{example}
\label{ex1} 
Let $\alpha=0.2$, $\mu=0.8$, $\tau=0.2$, $T=5$, 
$\omega(t)=\left(0.7\quad 0.7\right)^T$, 
and consider the homogeneous GPFS with time delay\index{time delay} given by 
\begin{equation}
\label{system2a}
\left\{
\begin{array}{ll}
{}^C \!D_{0}^{0.2,0.8}y(t)
=\exp\left(0.25t\right) \left[
\begin{pmatrix}
0.1&0\\
0&-0.2\\
\end{pmatrix}
\begin{pmatrix}
y_{1}(t)\\
y_{2}(t)\\
\end{pmatrix}
+
\begin{pmatrix}
-0.3&0\\
0&-0.2\\
\end{pmatrix}
\begin{pmatrix}
y_{1}(t-0.2)\\
y_{2}(t-0.2)\\	
\end{pmatrix}\right], 
& t\in[0,5],\\
y(t)=\left(0.7\quad 0.7\right)^T,  & t\in[-0.2,0].
\end{array}
\right.
\end{equation}	
According to system~\eqref{system2}, one has 
$$y=
\begin{pmatrix}
	y_{1}\\
	y_{2}\\	
\end{pmatrix}
\in \mathbb{R}^{2},~
A=\begin{pmatrix}
	0.1&0\\
	0&-0.2\\
\end{pmatrix}
\text{ and }
B=
\begin{pmatrix}
	-0.3&0\\
	0&-0.2\\
\end{pmatrix}.
$$

The aim is to verify condition~\eqref{cnd1} 
with respect to $\{c_{1}=1, c_{2}=8, T=5\}$. 

We have
$$
\Arrowvert \omega \Arrowvert_{C}=0.9899<c_{1},
$$	
$$
\Arrowvert A \Arrowvert= 0.2~ \text{ and }~ \Arrowvert B \Arrowvert=0.3.
$$
Therefore, condition~\eqref{cnd1} holds over the interval $[0, 5]$. 
Then, from Theorem~\ref{theob}, we deduce that system~\eqref{system2a} 
is finite time stable\index{finite time stable} with respect to $\{c_{1}=1, c_{2}=8, T=5\}$.	
\end{example}

% ----------------------------------

\begin{example} 
Let $\alpha=0.4$, $\mu=0.5$, $\tau=0.3$, $T=4$,
$\omega(t)=\left(0.5\tanh(t)\quad 0\right)^T$  
and consider the two-state fractional delayed nonlinear GPFS
\begin{equation}
\label{system2exp}
\left\{
\begin{array}{ll}
{}^C\!D_{0}^{0.4,0.5}y(t)=\exp\left(-t\right) 
\left[Ay(t)+By(t-0.3)+0.02(\tanh(y(t))+\tanh(y(t-0.3)))\right], 
& t\in[0,3],\\			
y(t)=\left(0.5\tanh(t)\quad 0\right)^T,  
& t\in[-0.3,0],
\end{array}
\right.
\end{equation}
with 
$$
A=\begin{pmatrix}
-0.2&0\\
0&0.3\\
\end{pmatrix},
\quad
B=
\begin{pmatrix}
0&0\\
0.3&0.4\\
\end{pmatrix}
$$
and 
$$
f(t,y(t),y(t-\tau))=0.02\left(\tanh(y(t))+\tanh(y(t-\tau))\right).
$$
One needs to check condition~\eqref{cnd1} with respect 
to $\{c_{1}=0.4, c_{2}=4, T=3\}$. 	

The nonlinear term $f$ satisfies the generalized Lipschitz 
condition~\eqref{condLipch}  with the constant $L_{f}=0.02$. Also, one has 
$$
\Arrowvert \omega \Arrowvert_{C}=0.2266	<c_{1},
$$	
$$
\Arrowvert A \Arrowvert= 0.3~ \text{ and }~ \Arrowvert B \Arrowvert=0.5.
$$	
Then, condition~\eqref{cnd1} holds over $[0,3]$. 
Consequently, from Theorem~\ref{theob}, we deduce 
the finite time stability\index{finite time stability} 
of system~\eqref{system2exp} with respect to 
$\{\xi=0.4, \varepsilon=4, T=3\}$.
\end{example}

% ---------------------------------------------------

\section{Conclusion}
\label{sec:5}

We have dealt with the problem of finite time stability for 
generalized proportional fractional systems (GPFSs) 
with time delay. A sufficient condition, which ensures 
the finite time stability for homogeneous delayed GPFSs, based 
on the generalized Gr\"onwall inequality, was obtained. Also, an explicit 
delay-dependent criterion that allows the stability over a finite time interval 
of a class of nonlinear GPFSs is provided.
The effectiveness of the proposed 
criteria has been illustrated by numerical examples.

% ---------------------------------------------------

\section*{Funding}

Zitane and Torres are supported by The Center for Research and
Development in Mathematics and Applications (CIDMA)
through the Portuguese Foundation for Science and Technology 
(FCT -- Funda\c{c}\~{a}o para a Ci\^{e}ncia e a Tecnologia),
projects UIDB/04106/2020 (\url{https://doi.org/10.54499/UIDB/04106/2020})
and UIDP/04106/2020 (\url{https://doi.org/10.54499/UIDP/04106/2020}).

% ---------------------------------------------------

% ---------------------------------------------------

\printindex


\begin{thebibliography}{99}
	
\bibitem{Almeida}
R. Almeida, R. Agarwal, S. Hristova\ and\ D. O'Regan,
Quadratic Lyapunov functions for stability of the generalized
proportional fractional differential equations with
applications to neural networks, 
Axioms {\bf 10} (2022), no.~4, Art.~322, 14~pp.

\bibitem{Gronwall} 
J. Alzabut, T. Abdeljawad, F. Jarad\ and\ W. Sudsutad,
A Gronwall inequality via the generalized proportional 
fractional derivative with applications, 
J. Inequal. Appl. {\bf 2019} (2019), Paper No. 101, 12~pp.

\bibitem{Arthi}
G. Arthi \ and\  N. Brindha,
On finite-time stability of nonlinear fractional-order systems 
with impulses and multi-state time delays,
Results in Control and Optimization {\bf 2} (2021), Paper No.~100010, 7~pp.
	
\bibitem{Holder}
E. F. Beckenbach\ and\ R. Bellman, 
{\it Inequalities}, 
Springer-Verlag, Inc., New York, 1983.	

\bibitem{Bohner}
M. Bohner\ and\ S. Hristova, 
Stability for generalized Caputo proportional 
fractional delay integro-differential equations, 
Bound. Value Probl. {\bf 2022} (2022), Paper No. 14, 15~pp. 

\bibitem{Boucenna}
D. Boucenna, D. Baleanu, A. Ben Makhlouf \ and\ A. M. Nagyf, 
Analysis and numerical solution of the generalized proportional fractional Cauchy problem,
Applied Numerical Mathematics {\bf 167} (2021), 173--186.

\bibitem{DuJia}
F. Du \ and\ B. Jia ,
Finite-time stability of a class of nonlinear fractional delay difference systems,
Appl. Math. Lett. {\bf 98} (2019), 233--239.

\bibitem{Du}
F. Du \ and\ J. G. Lu,
New criterion for finite-time stability of fractional delay systems,
Appl. Math. Lett. {\bf 104} (2020), Paper No.~106248, 7~pp.	
	
\bibitem{Farman}
M. Farman, A. Shehzad, A. Akg\"{u}l, D. Baleanu  \ and\ M. De la Sen, 
Modeling and analysis of a measles epidemic model with the constant proportional Caputo operator,
Symmetry {\bf 15} (2023), no.~2, Paper No.~468, 22~pp.

\bibitem{Gao}
Y. Gao \ and\ N. Li,
Fractional order PD control of the Hopf bifurcation of HBV viral systems with multiple time delays,
Alexandria Engineering Journal {\bf 83} (2023), 18~pp.

\bibitem{GPF} 
F. Jarad, T. Abdeljawad\ and\ J. Alzabut, 
Generalized fractional derivatives generated by a class of local proportional derivatives, 
Eur. Phys. J. Spec. Top. {\bf 226} (2017), no.~16, 3457--3471.	

\bibitem{Jia}
Y. Jia, C. Lin \ and\ B. Chen, 
Finite-time stability of singular time-delay systems based on a new weighted integral inequality,
Journal of the Franklin Institute {\bf 360} (2023), no.~7, 5092--5103.

\bibitem{Kuczma}
M. Kuczma, 
{\it An introduction to the theory of functional equations and inequalities}, 
Birkh\"{a}user Verlag, Basel, 2009.

\bibitem{Laadjal}
Z. Laadjal, T. Abdeljawad\ and\ F. Jarad, 
On existence-uniqueness results for proportional fractional differential 
equations and incomplete gamma functions, 
Adv. Difference Equ. {\bf 2020} (2020), Paper No. 641, 16~pp.

\bibitem{LaadjalJ}
Z. Laadjal \ and\  F. Jarad,
Existence, uniqueness and stability of solutions for generalized proportional
fractional hybrid integro-differential equations with Dirichlet boundary
conditions,
AIMS Mathematics {\bf 8} (2023), no.~1, 1172--1194.

\bibitem{Lazarevi}
M. P. Lazarevi\'{c}\ and\ A. M. Spasi\'{c}, 
Finite-time stability analysis of fractional order time-delay systems: 
Gronwall's approach, 
Math. Comput. Modelling. {\bf 49} (2009), no.~3-4, 475--481.

\bibitem{LiSun}
Q. Li, D. Sun, H. Liu\ and\ W. Zhao, 
Stability and bifurcation control of a delayed fractional 
eco-epidemiological system with saturated incidence,
Results in Physics {\bf 54} (2023), Paper No.~107019, 13~pp.

\bibitem{Li}
M. Li \ and\ J. Wang,
Finite time stability of fractional delay differential equations,
Applied Mathematics Letters {\bf 64} (2017), 170--176.

\bibitem{Nguyena}
T. T. H. Nguyena, N. T. Nguyen \ and\ M. N. Tran, 
Global fractional Halanay inequalities approach to finite-time stability
of nonlinear fractional order delay systems,
J. Math. Anal. Appl. {\bf 525} (2023), no.~1, Paper No.~127145, 16~pp.

\bibitem{Sabatier}
J. Sabatier, O. P. Agrawal \ and\ J. A. Machado,
Advances in Fractional Calculus,
Dordrecht: Springer Netherlands, 2007.

\bibitem{YangWu}
X. Yang, X. Wu \ and\ Q. Song,
Caputo-Wirtinger integral inequality and its application to stability 
analysis of fractional-order systems with mixed time-varying delays,
Applied Mathematics and Computation {\bf 460} (2024), Paper No.~128303, 12~pp.

\bibitem{YangZhang}
Z. Yang, J. Zhang, J. Hu\ and\ J. Mei,
New results on finite-time stability for fractional-order neural networks
with proportional delay,
Neurocomputing {\bf 442} (2021), 327--336.

\bibitem{MyID:549} 
H. Zitane\ and\ D. F. M. Torres,
Finite time stability of tempered fractional systems with time delays,
Chaos Solitons Fractals {\bf 177} (2023), Art.~114265, 10~pp.
{\tt arXiv:2311.06608}

\end{thebibliography}
\end{document}